\documentclass[12pt,final]{amsproc}

\usepackage[english]{babel}
\usepackage{latexsym}
\usepackage{amssymb}
\usepackage{amscd}
\usepackage[latin1]{inputenc}
\usepackage{mathrsfs}
\usepackage[notcite,notref]{showkeys}

\newtheorem{teor}{Theorem}[section]
\newtheorem{prop}[teor]{Proposition}

\newtheorem{lema}[teor]{Lemma}

\theoremstyle{definition}
\newtheorem{definition}{Definition}
\newtheorem{problem}{Problem}

\newtheorem{example}[teor]{Example}

 \newtheorem{defn}[teor]{Definition}
\newcommand{\adef}{\begin{defn}}
\newcommand{\zdef}{\end{defn}}

\theoremstyle{remark}

\newtheorem{remarks}[teor]{Remarks}

\newcommand{\R}{\mathbb{R}}
\newcommand{\N}{\mathbb{N}}

\newcommand{\U}{\mathscr{U}}

\newcommand{\aproof}{\begin{proof}}
\newcommand{\zproof}{\end{proof}}

\def\XiU{[X_i]_{\mathscr U}}
\def\Xii{(X_i)_{i\in I}}

\def\d{\operatorname{dist}}

\def\V{\mathscr{V}}

\def\e{\varepsilon}

\def\sg{\operatorname{sign}}
\def\N{\mathbb{N}}

\def\span{\operatorname{span}}
\def\R{\mathbb R}

\begin{document}

\title{On ultrapowers of Banach spaces of type $\mathscr L_\infty$}

\author[Avil\'es et al.]{Antonio Avil\'es,
F\'elix Cabello S\'anchez,\\ Jes\'us M. F. Castillo, Manuel
Gonz\'alez and Yolanda Moreno}

\address{Departamento de Matem\'aticas, Universidad de Murcia, 30100 Espinardo,
Murcia, Spain} \email{avileslo@um.es}

\address{Departamento de Matem\'aticas, Universidad de Extremadura, Avenida de Elvas s/n, 06071 Badajoz, Spain}
             \email{fcabello@unex.es}

\address{Departamento de Matem\'aticas, Universidad de Extremadura, Avenida de Elvas s/n, 06071 Badajoz, Spain}
             \email{castillo@unex.es}

\address{Departamento de Matem\'aticas, Universidad de Cantabria, Avenida los Castros s/n, 39071 Santander, Spain}
             \email{manuel.gonzalez@unican.es}

\address{Escuela Polit\'ecnica, Universidad de Extremadura, Avenida de la Universidad s/n, 10071 C\'aceres, Spain}
             \email{ymoreno@unex.es}





\bigskip

\bigskip

\maketitle

\begin{abstract} We prove that no ultraproduct of Banach spaces via a countably
incomplete ultrafilter can contain $c_0$ complemented. This shows
that a ``result'' widely used in the theory of ultraproducts is wrong.
We then amend a number of results whose proofs had been infected by that statement. In particular we provide proofs for the following statements:
(i) All $M$-spaces, in particular all $C(K)$-spaces, have
ultrapowers isomorphic to ultrapowers of $c_0$, as well as all
their complemented subspaces isomorphic to their square. (ii) No
ultrapower of the Gurari\u \i\ space can be complemented in any
$M$-space. (iii) There exist Banach spaces not complemented in
any $C(K)$-space having ultrapowers isomorphic to a
$C(K)$-space.\end{abstract}

\section{Introduction}
The Banach space ultraproduct construction has been, and still continues to be, the main bridge between model theory and the theory of Banach spaces and its ramifications. Ultraproducts of Banach spaces, even at a very elementary level, proved very useful in the ``local theory'', the study of Banach lattices, and also in some nonlinear problems, such as the uniform and Lipschitz classification of Banach spaces. We refer the reader to Heinrich's survey paper \cite{heinrich} and  Sims' notes \cite{sims} for two complementary accounts. While the study of the isometric properties of ultraproducts goes back to its inception in Banach space theory and produced  a rather coherent set of results
very early (see for instance \cite{hensoniovino}), not much is known about the isomorphic theory. The purpose of this paper is to study the interplay between the isomorphic theory of Banach spaces and ultraproducts, placing the emphasis on spaces of type $\mathscr L_\infty$. To do this we need first to clarify the status of a number of ``results'' in the theory of ultraproducts of Banach spaces.
Let us explain this point in detail as it might be the most interesting feature of the paper to some readers. We refer the reader to Section~\ref{pre} for precise definitions and all unexplained notation.

The following statement appears, without proof, as Lemma 4.2 (ii) in
Stern's paper \cite{stern}:
\begin{itemize}

\item[$\bigstar$] If $\mathscr U$ is a countably incomplete ultrafilter and $H$ is the corresponding ultrapower of $c_0$, then $H$ contains a complemented subspace isometric to $c_0(H)$.

\end{itemize}

Here, $c_0$ is the space of scalar sequences converging to zero and $c_0(H)$ is the space of sequences converging to zero in $H$, with the sup norm. This statement, however, turns out to be false (see below).
Unfortunately, Stern's Lemma has infected the proofs of a number
of results in the nonstandard theory and ultraproduct theory of Banach
spaces. We can mention:
\begin{itemize}

\item[(a)] If $E$ is isomorphic to a complemented subspace of a
$C$-space, then $E$ has an ultrapower
isomorphic to a $C$-space (Stern \cite[Theorem
4.5(ii)]{stern}) and also Henson-Moore \cite[Theorem 6.6
(c)]{hensonmoore83}).

\item[(b)] If $E$ is isomorphic to a complemented subspace of an
$M$-space, then $E$ has an ultrapower
isomorphic to a $C$-space (Heinrich-Henson
\cite[Theorem~12(c)]{HH}).

\item[(c)] If $E$ is an $M$-space then $E$ has an ultrapower isomorphic to an ultrapower
of $\ell_\infty$ (Henson-Moore \cite[Theorem
6.7]{hensonmoore83}).

\item[(d)] Ultrapowers of the Gurari\u \i\ space with respect to
countably incomplete ultrafilters are not complemented in any
$C$-space. (Henson-Moore \cite[Theorem~6.8]{hensonmoore83})
\end{itemize}

(Here, a $C$-space is a Banach space isometrically isomorphic to $C(K)$, the space of all continuous functions on the compact space $K$ with the sup norm again, while an $M$-space is a sublattice of a $C$-space; see Section~\ref{lind}.)\\

With this background in mind let us explain the organization of the paper and summarize its main results. Section~\ref{pre} is preliminary and it mostly consists of definitions and conventions about the notation. Section~\ref{stern} contains a few general results on the structure of ultraproducts of Banach spaces -- we invariably assume they are built over countably incomplete ultrafilters. We will show that ultraproducts of Banach spaces are Grothendieck spaces as long as they are $\mathscr L_\infty$-spaces (Proposition \ref{ultra-Linf}). A Grothendieck space is a Banach space were $c_0$-valued operators are weakly compact: in particular no Grothendieck space can contain a complemented copy of $c_0$. This already shows that Stern's lemma is wrong. And indeed more is true: $c_0$ is never complemented in ultraproducts (Proposition \ref{coout}). Interesting sideways can
be taken to arrive to these results. In \cite{accgm1} we have shown that ultrapowers of $\mathscr L_\infty$-spaces are
``universally separably injective" --$E$ is universally separably
injective if $E$-valued operators extend from separable subspaces-- and that universally
separably injective spaces are always Grothendieck. To complete those results we have added a proof that infinite dimensional ultraproducts
via countably incomplete ultrafilters are never injective spaces,
a result basically due to Henson and Moore \cite[Theorem
2.6]{hensonmoore}. In Section~\ref{u} we consider the problem of whether two given Banach spaces have isomorphic (not necessarily isometric) ultrapowers. Regarding the statements (a) to (d) we show that (c) and (d) are true and we provide amendments for (a) and (b) by proving that they hold
under the additional hypothesis
that $E$ is isomorphic to its square. The closing Section~\ref{open} contains some additional results, together with some open problems that we found interesting.

\section{Preliminaries}\label{pre}

\subsection{Filters} A family $\mathscr U$ of subsets of a given set $I$ is said to be a filter if it is closed under finite intersection, does not contain the empty set and, one has $A\in\mathscr U$ provided $B\subset A$ and $B\in\mathscr U$. An ultrafilter on $I$ is a filter which is maximal with respect to inclusion. If $X$ is a (Hausdorff) topological space, $f:I\to X$ is a function, and $x\in X$, one says that $f(i)$ converges to $x$ along $\mathscr U$ (written $x=\lim_\mathscr U f(i)$ to short) if whenever $V$ is a neighborhood of $x$ in $X$ the set $f^{-1}(V)=\{i\in I: f(i)\in V\}$ belongs to $\mathscr U$. The obvious compactness argument shows that if $X$ is compact and Hausdorff, and $\mathscr U$ is an ultrafilter on $I$, then for every function $f:I\to X$ there is a unique $x\in X$ such that $x=\lim_\mathscr U f(i)$.

\begin{definition}\label{CI}
An ultrafilter $\U$ on a set $I$ is countably incomplete if there
is a sequence $(I_n)$ of subsets of $I$ such that
$I_n\in \U$ for all $n$, and $\bigcap_{n=1}^\infty
I_n=\varnothing$.
\end{definition}

Throughout this paper all ultrafilters will be assumed to be
countably incomplete. Notice that $\U$ is countably incomplete if
and only if there is a function $n:I\to \N$ such that
$n(i)\to\infty$ along $\U$ (equivalently, there is a family
$\e(i)$ of strictly positive numbers converging to zero along
$\U$). It is obvious that any countably incomplete ultrafilter is free (it contains no singleton)
and also that every  free
ultrafilter on $\N$ is countably incomplete. Assuming all free
ultrafilters are countably incomplete is consistent with
{\textsf{ZFC}}, the usual setting of set theory, with the axiom of choice.

\subsection{Ultraproducts of Banach spaces} 
Let us briefly recall the definition and some basic properties of
ultraproducts of Banach spaces. Let $\Xii$ be a family of Banach spaces indexed by the set $I$ and let $\U$ be an ultrafilter on $I$. The space
of bounded families $ \ell_\infty(I, X_i)$ endowed
with the supremum norm is a Banach space, and $ c_0^\U(X_i)=
\{(x_i) \in \ell_\infty(I,X_i) : \lim_{\U} \|x_i\|=0\} $ is a
closed subspace of $\ell_\infty(I, X_i)$. The ultraproduct of the
spaces $\Xii$ following $\U$ is defined as the quotient space
$$
[X_i]_\U = {\ell_\infty(I, X_i)}/{c_0^\U(X_i)},
$$
with the quotient norm.
We denote by $[(x_i)]$ the element of $[X_i]_\U$ which has the
family $(x_i)$ as a representative. It is easy to see that $ \|[(x_i)]\| = \lim_{\U} \|x_i\|. $ In the case $X_i = X$
for all $i$, we denote the ultraproduct by $X_\U$, and call it the
ultrapower of $X$ following $\U$. If $T_i:X_i\to Y_i$ is a
uniformly bounded family of operators, the ultraproduct operator
$[T_i]_\U: [X_i]_\U\to [Y_i]_\U$ is given by $[T_i]_\U[(x_i)]=
[T_i(x_i)]$. Quite clearly, $ \|[T_i]_\U\|= \lim_{\U}\|T_i\|. $
\subsection{Banach spaces of type $\mathscr L_\infty$ and Lindenstrauss spaces}
Throughout the paper we shall write $X\sim Y$ to indicate that the Banach spaces $X$ and $Y$ are linearly isomorphic. If they are isometric we write $X\approx Y$. The ground field is $\R$.

A Banach space $X$ is said to be an
$\mathscr L_{\infty,\lambda}$-space (with $\lambda \geq 1$) if every finite dimensional subspace $F$ of $X$
is contained in another finite dimensional subspace of $X$ whose
Banach-Mazur distance to the corresponding $\ell_\infty^n$ is at most
$\lambda$. The Banach-Mazur distance between two isomorphic Banach spaces $X$ and $Y$ is defined as $d(X,Y)=\inf_T\|T\|\|T^{-1}\|$, where $T$ runs over all isomorphisms between $X$ and $Y$.

An $\mathscr L_\infty$-space is just a $\mathscr{L}_{\infty,\lambda}$-space for some $\lambda\geq 1$; we
will say that it is a $\mathscr{L}_{\infty,\lambda+}$-space when it is
a $\mathscr{L}_{\infty,\mu}$-space for all $\mu>\lambda$. The
$\mathscr{L}_{\infty,1+}$-spaces are usually called Lindenstrauss
spaces and coincide with the isometric preduals of $L_1(\mu)$-spaces; see \cite[Theorem 4.1]{zippin}. The classes of
$\mathscr L_{\infty,\lambda+}$ spaces are stable under ultraproducts
\cite[Proposition 1.22]{BourgainLNM}. In the opposite direction, a
Banach space is a $\mathscr{L}_{\infty,\lambda+}$ space if and only
if some (or every) ultrapower is. In particular, a Banach space is
a $\mathscr{L}_{\infty}$ space or a Lindenstrauss  space if and
only if so are its ultrapowers; see, e.g., \cite{heinrichL1}. However it is possible to obtain Lindenstrauss\ spaces as ultraproducts of families of reflexive spaces: indeed, if $p(i)\to\infty$
along $\U$, then the ultraproduct $[L_{p(i)}]_\U$ is a
Lindenstrauss\ space -- in fact, an abstract $
M$-space; see \cite[Lemma 3.2]{c:M}.

\subsection{Some classes of Lindenstrauss spaces}\label{lind}
Some distinguished classes of Lindenstrauss spaces we shall consider along the paper are:
\begin{itemize}
\item $C$-spaces: Banach spaces of the form $C(K)$ for some compact Hausdorff space $K$, with the sup norm.
\item $C_0$-spaces: maximal ideals of $C$-spaces.
\item $G$-spaces: Banach spaces of the form $X=\{f\in C(K): f(x_i)=\lambda f(y_i) \text{ for all }i\in I\}$ for some compact space $K$ and some family of triples $(x_i,y_i,\lambda_i)$, where $x_i, y_i\in K$ and $\lambda_i\in \mathbb R$.
\item $M$-spaces: $G$-spaces where $\lambda_i\geq 0$ for every $i\in I$; equivalently, the closed sublattices of the $C$-spaces.
\end{itemize}

It is perhaps worth noticing that all these classes admit quite elegant characterizations: $C_0$-spaces ($C$-spaces) are exactly those real Banach algebras $X$ (with unit) satisfying the inequality $\|x\|^2\leq \|x^2+y^2\|$ for all $x,y\in X$, a classical result by Arens; see  \cite[Theorem 4.2.5]{a-k}.
Also, a Banach lattice $X$ is representable as a concrete $M$-space if and only if one has $\|x+y\|=\max(\|x\|,\|y\|)$ whenever $x$ and $y$ are disjoint, that is $|x|\wedge|y|=0$. Finally, $G$-spaces are exactly those Banach spaces that are contractively complemented in $M$-spaces. The preceding classes are closed under ultraproducts, see \cite[Proposition 1]{heinrichL1}. In particular, if $(K_i)$ is a family of compact spaces indexed by $I$ and $\U$ is an ultrafilter on $I$, then there is a compact space $K$ such that $[C(K_i)]_\U$ is isometric to $C(K)$. This compact space $K$ is often called the ultracoproduct of the family $(K_i)$ with respect to $\U$ and its is denoted by $(K_i)^\U$. We refer the interested reader to \cite[Section 4]{heinrich} or \cite[Section 8]{sims} for a description of $(K_i)^\U$ based on Banach algebras techniques and to
 \cite[Section 5]{bankston-survey} for a purely topological construction of the ultracoproduct.

\section{Around Stern's lemma}\label{stern}
Throughout this Section $\XiU$ will denote the ultraproduct of a family of Banach spaces $\Xii$ with respect to a countably incomplete ultrafilter $\U$.
We begin with the following result about the structure of separable
subspaces of ultraproducts of type $\mathscr L_\infty$.

\begin{lema}\label{S}
Supppose $\XiU$ is an $\mathscr L_{\infty, \lambda+}$-space. Then
each separable subspace of $\XiU$ is contained in a subspace of
the form $[F_i]_\U$, where $F_i\subset X_i$ is finite dimensional
and $\lim_{\U(i)}d(F_i,\ell_\infty^{k(i)})\leq \lambda$, with
$k(i)=\dim F_i$.
\end{lema}

\begin{proof} Let us assume $S$ is an infinite-dimensional separable subspace of $\XiU$.
Let $(s^n)$ be a linearly independent sequence spanning a dense
subspace in $S$ and, for each $n$, let $(s_i^n)$ be a fixed
representative of $s^n$ in $\ell_\infty(I,X_i)$. Let
$S^n=\span\{s^1,\dots,s^n\}$. Since $\XiU$ is an $\mathscr L_{\infty,\lambda+}$-space there is, for each $n$, a finite
dimensional $F^n\subset \XiU$ containing $S^n$ with
$d(F^n,\ell_\infty^{\dim F^n })\leq \lambda+1/n$. For fixed $n$, let $(f^m)$ be a basis for $F^n$ containing
$s^1,\dots, s^n$. Choose representatives $(f_i^m)$ such that
$f_i^m=s_i^\ell$ if $f^m=s^\ell$. Moreover, let $F_i^n$ be the
subspace of $X_i$ spanned by $f_i^m$ for $1\leq m\leq \dim F^n$. Let $(I_n)$ be a decreasing sequence of subsets $I_n\in \U$ such
that  $\bigcap_{n=1}^\infty I_n=\varnothing$. For each integer $n$
put$$ J'_n=\left\{i\in I: d\left(F_i^n, \ell_\infty^{\dim F^n}\right) \leq
\lambda+2/n \right\} \cap I_n
$$
and $J_m=\bigcap_{n\leq m}J'_n$. All these sets are in $\U$. We
define a function $k:I\to \N$ as
$$
k(i)= \sup\{n:i\in J_n\}.
$$For each $i\in I$, take $F_i=F_i^{k(i)}$. This is a
finite-dimensional subspace of $X_i$ whose Banach-Mazur distance
to the corresponding $\ell_\infty^k$ is at most $\lambda+2/k(i)$.
It is clear that $[F_i]_{\mathscr U}$ contains $S$ and also that
$k(i)\to\infty$ along $\U$, which completes the proof.
\end{proof}

Recall that a Banach space $X$ is said to be a Grothendieck space
if every $c_0$-valued operator is weakly compact; equivalently, if
weak* and weak convergence for sequences in the dual space
coincide. For every set $\Gamma$ the space $\ell_\infty(\Gamma)$
is Grothendieck. One has:

\begin{prop}\label{ultra-Linf}
If $\XiU$ is an
$\mathscr L_\infty$-space, then it is a Grothendieck space.
\end{prop}
\begin{proof} It is fairly obvious that a Banach space $X$ in
which every separable subspace is contained in a Grothendieck
subspace of $X$ must be a Grothendieck space. Thus, in view of Lemma \ref{S}, everything one
needs is to show that all spaces $[\ell_\infty^{n(i)}]_{\mathscr
U}$ are Grothendieck spaces. But this follows from the definition
of the ultraproduct space as $[\ell_\infty^{n(i)}]_{\mathscr U}$
as a quotient of $\ell_\infty(\ell_\infty^{n(i)}) =
\ell_\infty(\Gamma)$ and the simple fact that quotients of Grothendieck
spaces are Grothendieck spaces.
\end{proof}

Therefore,  ultraproducts which are $\mathscr L_{\infty}$-spaces
cannot contain infinite dimensional separable complemented subspaces, in particular,
$c_0$. This shows that Stern's claim that $c_0((c_0)_{\mathscr
U})$ is isometric to a complemented subspace of $(c_0)_{\mathscr
U}$ cannot be true since $c_0((c_0)_{\mathscr U})$ obviously contains
complemented copies of $c_0$. If we focus our attention on copies of $c_0$, we
can present a much more general result, which improves Corollary
3.14 of Henson and Moore in \cite{hensonmoore83}.

\begin{prop}\label{coout}
No ultraproduct of Banach spaces over a countably incomplete
ultrafilter contains a complemented subspace isomorphic to $c_0$.
\end{prop}

\begin{proof}
Assume
$[X_i]_\U$ has a subspace isomorphic to $c_0$, complemented or
not, and let $\imath:c_0\to [X_i]_\U$ be the corresponding
embedding.

Let $f^n=\imath(e_n)$, where $(e_n)$ denotes the traditional basis
of $c_0$, and let $(f_i^n)$ be a representative of $f^n$ in
$\ell_\infty(I, X_i)$, with $\|(f_i^n)\|_\infty=\|f^n\|$. Then we
have
$$
\|\imath^{-1}\|^{-1}\|(t_n)\|_\infty\leq \|\sum_nt_n f^n\|\leq
\|\imath\|\|(t_n)\|_\infty,
$$
for all $(t_n)$ in $c_0$. Fix $0<c<\|\imath^{-1}\|^{-1}$ and
$\|\imath\|<C$ and, for $k\in \N$ define
$$
J_k=\left\{i\in I: c\|(t_n)\|_\infty\leq \|\sum_{n=1}^k
t_nf^n_i\|_{X_i}\leq C\|(t_n)\|_\infty \text{ for all }
(t_n)\in\ell_\infty^k\right\}.
$$

It is easily seen that $J_k$ belongs to $\U$ for all $k$.
Moreover, $J_1=I$ and $J_{k+1}\subset J_k$ for all $k\in N$. Now,
for each $i\in I$, define $k:I\to\N\cup\{\infty\}$ taking
$k(i)=\sup\{n:i\in J_n\}$.

Let us consider the ultraproduct $[c_0^{k(i)}]_\U$, where
$c_0^k=\ell_\infty^k$ when $k$ is finite and $c_0^k=c_0$ for
$k=\infty$. We define operators $\jmath_i:c^{k(i)}_0\to X_i$
taking $\jmath_i(e_n)=f^n_i$ for $1\leq n\leq k(i)$ for finite
$k(i)$ and for all $n$ if $k(i)=\infty$. These are uniformly
bounded and so they define an operator $\jmath:[c_0^{k(i)}]_\U\to
\XiU$. Also, we define $\kappa:c_0\to [c_0^{k(i)}]_\U$ taking
$\kappa(x)=[(\kappa_i(x))]$, where $\kappa_i$ is the obvious
projection of $c_0$ onto $c_0^{k(i)}$. We claim that
$\jmath\kappa=\imath$. Indeed, for $n\in\N$, we have
$\kappa_i(e_n)=e_n$ (at least) for all $i\in J_n$ and since
$J_n\in \U$ we have $\jmath\circ \kappa(e_n)= \imath(e_n)$ for all
$n\in\N$. Now, if $p:\XiU\to c_0$ is a projection for $\imath$,
that is, $p\imath$ is the identity on ${c_0}$, then $p\jmath$ is a projection
for $\kappa:c_0\to [c_0^{k(i)}]_\U$, which cannot be since the
latter is a Grothendieck space.
\end{proof}

As we mentioned in the introduction, if an ultraproduct $E$ is an $\mathscr L_\infty$-space then it is
 universally separably injective in the following sense: for
every Banach space $X$ and each separable subspace $Y\subset X$,
every operator $t:Y\to E$ extends to an operator $T:X\to E$; see \cite[Theorem 4.10]{accgm1}. In
spite of this fact, infinite dimensional
ultraproducts via a countably incomplete ultrafilters are never
injective (a Banach space $E$ is said to be injective when
$E$-valued operators can be extended to any superspace). We give
the proof here because Henson-Moore proof in \cite[Theorem
2.6]{hensonmoore} is written in the language of nonstandard
analysis and Sims' version for ultraproducts along Section~8 of
\cite{sims} is not very accessible.

\begin{teor}[Henson and Moore]\label{neverinj}
Ultraproducts via countably incomplete ultrafilters are never
injective, unless they are finite dimensional.
\end{teor}

\begin{proof}
Recalling that injective Banach spaces are $\mathscr L_\infty$-spaces, assume  that $\XiU$ is a $\mathscr
L_\infty$-space. According to Lemma~\ref{S}, if $\XiU$ is infinite
dimensional, it contains some infinite dimensional complemented
subspace isomorphic to $[\ell_\infty^{k(i)}]_{\U}$. Thus, it suffices to see that the later is not an injective space.

Let $(S_i)_{i\in I}$ be a family of sets and $\U$ an ultrafilter
on $I$. The set-theoretic ultraproduct $\langle S_i\rangle_\U$ is
the product set $\prod_iS_i$ factored by  the equivalence relation
$$
(s_i)\equiv(t_i)\Longleftrightarrow \{i\in I:s_i=t_i\}\in \U.
$$
The class of $(s_i)$ in $\langle S_i\rangle_\U$ is denoted
$\langle(s_i)\rangle$. Let thus $\langle \{1, \dots ,
k(i)\}\rangle_\U$ denote the set-theoretic ultraproduct of the
sets $\{1, \dots , k(i)\}$.  We have
\begin{equation}\label{inclusions}
c_0(\langle \{1, \dots , k(i)\}\rangle_\U )\subset
[\ell_\infty^{k(i)}]_{\U}\subset \ell_\infty(\langle \{1, \dots ,
k(i)\}\rangle_\U ).
\end{equation}
This should be understood as follows: each $[(f_i)]\in [\ell_\infty^{k(i)}]_{\U} $ defines a
function on $\langle \{1, \dots , k(i)\}\rangle_\U$ by the formula
$f\langle (x_i)\rangle_{\U}=\lim_{\U(i)} f_i(x_i)$. In this way,
$[\ell_\infty^{k(i)}]_{\U}$ embeds isometrically as a subspace of
$\ell_\infty(\langle \{1, \dots , k(i)\}\rangle_\U)$ containing $
c_0(\langle \{1, \dots , k(i)\}\rangle_\U)$. Write $\Gamma=\langle \{1, \dots , k(i)\}\rangle_\U$ and
$U=[\ell_\infty^{k(i)}]_\U$, so that (\ref{inclusions}) becomes
$c_0(\Gamma)\subset U\subset \ell_\infty(\Gamma)$. We will prove
that the inclusion of $c_0(\Gamma)$ into $U$ cannot be extended to
$ \ell_\infty^c(\Gamma)$,  the space of all countably supported
bounded families on $\Gamma$.

Recall that an internal subset of $\Gamma$ is one of the form
$\langle A_i\rangle_\U$, where $A_i\subset \{1,\dots,k(i)\}$ for
each $i\in I$. Infinite internal sets  must have cardinality at
least $\frak c$  ---just use an almost disjoint family. This is
the basis of the ensuing argument: as $U$ is spanned by the
characteristic functions of the internal sets, if $f\in U$ is not
in $c_0(\Gamma)$, then there is $\delta>0$ and an infinite
internal $A\subset \Gamma$ such that $|f|\geq \delta$ on $A$.

Suppose $I:\ell_\infty^c(\Gamma)\to
U$ is an operator extending the inclusion of $c_0(\Gamma)$ into
$U$. Given a countable $S\subset \Gamma$, let us consider
$\ell_\infty(S)$ as the subspace of $\ell_\infty^c(\Gamma)$
consisting of all functions vanishing outside $S$ and let us write
$I_S$ for the endomorphism of $\ell_\infty(S)$ given by $I_S(f)=
1_S I(f)$, where $1_S$ is the characteristic function of $S$.
Notice that $I_S$ cannot map $\ell_\infty(S)$ to
$c_0(S)$ since $c_0$ is not complemented in $\ell_\infty$. Thus,
given an infinite countable $S\subset \Gamma$, there is  a norm
one $f\in \ell_\infty(S)$ (the characteristic function of a
countable subset of $S$, if you prefer), a number $\delta>0$ and
an infinite internal $A\subset\Gamma$ such that $|I(f)|\geq
\delta$ on $A$, with $|A\cap S|=\aleph_0$. Let $\beta(S)$ denote
the supremum of the numbers $\delta$ arising in this way. Also, if
$T$ is any subset of $\Gamma$, put $\beta[T]=\sup\{\beta(S):
S\subset T, |S|=\aleph_0\}$.

Let $S_1$ be a countable set such that $\beta(S_1)>{1\over
2}\beta[\Gamma]$ and let us take $f_1\in\ell_\infty(S_1)$ such that
$|I(f_1)|> {1\over 2}\beta(S_1)$ on an infinite internal set $A^1$
with $|A^1\cap S_1|=\aleph_0$.

Let $S_2$ be a countable  subset of $A^1\backslash S_1$ (notice $|A^1\backslash S_1|\geq \frak c$) such that $\beta(S_2)>{1\over 2}\beta[A^1\backslash S_1]$ and take a norm one $f_2\in \ell_\infty(S_2)$ such that $|I(f_2)|\geq {1\over 2}\beta(S_2)$ on an infinite internal set $A^2\subset A^1$ with $|A^2\cap S_2|=\aleph_0$. 

Let $S_3$ be an infinite countable subset of $A^2\backslash
(S_1\cup S_2)$ such that $\beta(S_3)>{1\over 2}\beta[A^2\backslash
(S_1\cup S_2)]$ and take a normalized $f_3\in\ell_\infty(S_3)$
such that $|If_3|>{1\over 2}\beta(S_3)$ on certain internal
$A^3\subset A^2$ such that $|A^3\cap S_3|=\aleph_0$ and so on.

Continuing in this way we get sequences $(S_n), (f_n)$ and $
(A^n)$, where
\begin{itemize}
\item Each $A^n$ is an infinite internal subset of $\Gamma$. \item
$A^0=\Gamma$ and  $A^{n+1}\subset A^n$ for all $n$. \item
$S_{n+1}$ is a countable subset of $A^{n}\backslash
\bigcup_{m=1}^n S_m$, and $\beta(S_{n+1})>{1\over
2}\beta[A^{n}\backslash \bigcup_{m=1}^n S_m]$. \item $f_n$ is a
normalized function in $\ell_\infty(S_n)$. \item $|If_n|> {1\over
2}\beta(S_n)$ on $A^n$. \item For each $n$ one has $|A^n\cap
S_n|=\aleph_0$.
\end{itemize}
Our immediate aim is to see that $\beta(S_n)$ converges to zero.
Fix $n$ and consider any $a\in A^{n+1}$ to define
$$
h_n=\sum_{m= 1}^n \sg(If_m(a)) f_m.
$$
Clearly, $\|h_n\|= 1$ since the $f_m$'s have disjoint supports. On
the other hand,
$$
\|I\|\geq \|Ih_n\|\geq Ih_n(a)=\sum_{m=1}^n |If_m(a)|\geq
\frac{1}{2}\sum_{m=1}^n\beta(S_m),
$$
so $(\beta(S_n))$ is even summable.

For each $n\in\N$, choose a point $a_n\in S_n$ and consider the
set $S=\{a_n:n\in \N\}$. We achieve the final contradiction by
showing that $I_S$ maps $\ell_\infty(S)$ to $c_0(S)$, thus
completing the proof. Indeed, pick $f\in\ell_\infty(S)$ and let us
compute $\d(1_SI(f),c_0(S))$. For each $n\in\N$, set
$R_n=\{a_m:m\geq n\}$. We have $f=1_{R_n}f+(1_S-1_{R_n})f$ and
since $S\backslash R_n$ is finite,
$If=I1_{R_n}f+I((1_S-1_{R_n})f)=I1_{R_n}f+(1_S-1_{R_n})f$.
Moreover, the function $1_{R_n}f$ has countable support contained
in  $A^{n}\backslash \bigcup_{m=1}^n S_m$. So,
$$
\begin{aligned}
\d(1_SIf,c_0(S))&= \d(1_SI1_{R_n}f,c_0(S))\\
&\leq  \d(1_{R_n}I1_{R_n}f,c_0(R_n))+\d(1_{S\backslash R_n}I1_{R_n}f,c_0(S\backslash R_n))\\
&=  \d(1_{R_n}I1_{R_n}f,c_0(R_n))\\
&\leq \|1_{R_n}f\|\beta(R_n)\\
&\leq \|f\|\beta\left[ A^{n}\backslash \bigcup_{m=1}^n S_m \right]\\
&\leq 2\|f\|\beta(S_{n+1}).
\end{aligned}
$$
And since $\beta(S_{n+1})\to 0$ we are done.
\end{proof}

\begin{remarks}
(a) Let us give a simpler proof of Theorem~\ref{neverinj} for ``countable'' ultraproducts. The ensuing argument relies on
Rosenthal's result \cite[Corollary 1.5]{rose} asserting that an injective Banach
space containing $c_0(\Gamma)$ contains $\ell_\infty(\Gamma)$ as
well. Suppose $I$ countable. Then $[\ell_\infty^{k(i)}]_\U$ is a
quotient of $\ell_\infty$, and so its density character is (at
most) the continuum. On the other hand, if
$[\ell_\infty^{k(i)}]_\U$ is infinite dimensional, then
$\lim_{\U(i)}k(i)=\infty$, and using an almost disjoint family we
see that the cardinality of $\Gamma=\langle \{1, \dots ,
k(i)\}\rangle_\U$ equals the continuum. Thus, if
$[\ell_\infty^{k(i)}]_\U$ were injective, as it contains $c_0(\Gamma)$,  it should contain a copy
of $\ell_\infty(\langle \{1, \dots , k(i)\}\rangle_\U)$, which is
not possible, because the later space has density character
$2^\mathfrak c$.
\smallskip

(b) Leung and
R\"abiger proved in \cite{leura} that given a family
$(E_i)_{i\in I}$ of Banach spaces containing no complemented copy of
$c_0$, the space $\ell_\infty(I,E_i)$ does not contain a
complemented copy of $c_0$ if the cardinal of $I$ is not real-valued measurable (that is, every countably additive measure defined on the power set of $I$ and vanishing on every singleton is zero), in particular if $I$ is countable. This implies that
when $I$ has non-real-valued measurable cardinal, the ultraproduct
$(E_i)_{\mathscr U}$ of a family $(E_i)_{i\in I}$ of Lindenstrauss
Grothendieck spaces is a Grothendieck space.
\smallskip

(c) It is a challenging problem in set theory to decide if measurable cardinals exist, that is, if some set can ever support a countably  complete, free ultrafilter. In any case such a cardinal should be very, very large; see \cite[Section 4.2]{c-k}. However ultraproducts based on countably complete ultrafilters should not be very interesting to us. In fact, if $\U$ is countably complete and $|X|$ is less that the least uncountable measurable cardinal, then $X_\U=X$ in the sense that the diagonal embedding is onto. This is so because if $\U$ is countably complete, one has $\langle X_i\rangle_\U= \XiU$ for all families of Banach spaces in view of the
remark following Definition~\ref{CI} and the diagonal embedding of $X$ into $\langle X\rangle_\U$ is onto according to \cite[Corollary 4.2.8]{c-k}.
\end{remarks}

\section{Isomorphic equivalence}\label{u}

As we mentioned before, the study of the isometric equivalence of ultrapowers goes back to the inception of the ultraproduct construction in Banach space theory and has produced many interesting results in the ``model theory of Banach spaces''. In this Section we will rather consider the isomorphic variation introduced by Henson and Moore
\cite[p.106]{hensonmoore83}.

\begin{definition}
We say that two Banach spaces $X$ and $Y$ are ultra-isomorphic (respectively, ultra-isometric) and we write $X\stackrel{u}\sim Y$ (respectively, $X\stackrel{u}\approx Y$) to short
if there is an ultrafilter $\mathscr U$ such that $X_\U$ and $Y_\U$ are isomorphic (respectively, isometric).
\end{definition}

Sometimes we will say that $X$ and $Y$ have the same ultratype. The following observation shows that  ``having the same
ultratype''  provides a true equivalence relation.

\begin{lema}
$X$ and $Y$ are ultra-isomorphic if (and only if) there are ultrafilters $\mathscr U$ and $\mathscr V$ such that $X_\mathscr U$ and $Y_\mathscr V$ are isomorphic.
\end{lema}

\begin{proof} The iteration of
ultrapowers produces new ultrapowers. Indeed, suppose that $\U,\V$ are ultrafilters on $I$ and $J$ respectively. Let $\mathscr W$ denote the family of all
subsets $W$ of $K=I\times J$ for which the set $\{j\in J: \{i\in I: (i,j)\in W\}\in \U\}$ belongs to $\V$. Then $\mathscr W$ is an ultrafilter, often denoted by $\U\times \V$, and moreover, one has $Z_\mathscr W=(Z_\U)_\V$ for all Banach spaces $Z$.
On the other hand, the Banach space version of the Keisler-Shelah isomorphism
theorem due to Stern \cite[Theorem 2.1]{stern}
establishes that given a Banach space $X$ and two ultrafilters $\mathscr U, \mathscr V$  then there is an ultrafilter $\mathscr W$ on some index set $K$ such that $(X_\U)_\mathscr W\approx (X_\V)_\mathscr W$.

Now, if $X_{\mathscr U} \sim Y_{\mathscr
V}$, taking
an ultrafilter $\mathscr W$ such that $(Y_{\mathscr U})_\mathscr W\approx (Y_\mathscr V)_\mathscr W$ we have
$$X_{\mathscr U \times \mathscr W} =  (X_{\mathscr U})_{\mathscr W} \sim (Y_{\mathscr V})_{\mathscr
W} \approx (Y_{\mathscr U})_\mathscr W =  Y_{\mathscr U \times
\mathscr W}.$$
\\[-30pt]
\end{proof}

Recall that a Banach space is an $\mathscr L_\infty$-space if and only if some (or every) ultrapower is.
The question of the classification of
$\mathscr L_\infty$-spaces appears posed in \cite[p. 106]{hensonmoore83} and \cite[p. 315]{HH} and was considered in \cite{henson-76}

\begin{problem}
How
many ultra-types of $\mathscr L_\infty$-spaces are there?
\end{problem}

We will support Henson-Moore assertion \cite[p. 106]{hensonmoore83}
that there are at least two different ultra-types:
one is that of $C$-spaces and the other is that of Gurari\u\i\
space. The following result was proved by Henson long time ago \cite[Corollary 3.11]{henson-76} for nonstandard hulls of $C$-spaces (instead of ultrapowers of $M$-spaces).
We give a proof based on ideas of \cite{stern} that can be easily modified to the effect of proving next Theorem~\ref{uplus}. To simplify the exposition let
us write $X\vartriangleleft Y$ to mean that $X$ is isomorphic to a
complemented subspace of $Y$.

\begin{prop}\label{Mc0}
All infinite dimensional $M$-spaces have the same ultra-type.
\end{prop}
\begin{proof}
The key of the reasoning is the
following nice result of Stern \cite[Theorem 2.2]{stern}: \emph{Let  $F$ be  a separable subspace of the Banach space $E$. There exists
a separable subspace $L$ of $E$ containing $F$ and an ultrafilter $\U$ such that $L_\U\approx E_\U$. If $E$ is a Banach lattice then $L$ can be
chosen to be a sublattice of $E$.} This implies that
every $M$-space $X$ has an ultrapower isometric to an ultrapower of some separable $M$-space $Y$. It is therefore enough to prove the
assertion for separable $M$-spaces and we will prove that if $X$ is an infinite dimensional separable $M$-space, then $X\stackrel{u}\sim c_0$.

We first observe that $c_0 \vartriangleleft X$: all Lindenstrauss spaces contain copies of $c_0$ and all copies of $c_0$ are complemented in separable spaces. On the other hand, by the very definition of a separable $\mathscr L_\infty$-space we see that $X$ embeds into an ultraproduct $(\ell_\infty^n)_\U$, where $\U$ is any free ultrafilter on the integers. Therefore $X$ embeds as a subspace of $(c_0)_\U$. By Stern's result quoted above there is a separable sublattice $L$ of $(c_0)_\U$ which contains a copy of $X$ and an ultrafilter $\V$ such that $L_\V\approx (c_0)_{\U\times\V}$. But $X$ and $L$ are $M$-spaces and separable $M$-spaces are isomorphic to $C$-spaces (Benyamini \cite{beny}). This implies that:
\begin{itemize}
\item $X$ is isomorphic to its square (Bessaga-Pe\l czy\'nski \cite[Theorem 3]{besapelc});
\item $L$ contains a complemented copy of $X$ (Pe\l czy\'nski \cite[Theorem 1]{pelc}) .
\end{itemize}
We have arrived to the following situation:
$$
X_\V  \vartriangleleft L_\V \approx (c_0)_{\U\times\V}  \vartriangleleft X_{\U\times\V}.
$$
Now we can apply the ultrapower theorem to get an ultrafilter $\mathscr W$ such that $(X_\V)_\mathscr W\approx  (X_{\U\times\V})_\mathscr W$. Letting $\mathscr T= \mathscr{(U\times V)\times W}$ we have
$$
X_\mathscr T\approx (X_\mathscr{V})_\mathscr W \vartriangleleft  ((c_0)_{\U\times\V})_\mathscr W = (c_0)_\mathscr T.
$$
 Recalling that $c_0  \vartriangleleft X$ one also has $(c_0)_\mathscr T  \vartriangleleft X_\mathscr T$.
Since both spaces $X$ and $c_0$ are isomorphic to their squares the same is true for their ultrapowers and Pe\l czy\'nski's
decomposition method (see \cite{pelcde}) yields $X_\mathscr T\approx  (c_0)_\mathscr T$.
\end{proof}

\begin{teor}\label{uplus} Let $X$ be either an $M$-space
or a complemented subspace of an $M$-space that is
moreover isomorphic to its square.
Then $X\stackrel{u}\sim\ell_\infty$.
\end{teor}

\begin{proof} If $X$ is an $M$-space the statement is contained in the preceding Proposition.
Suppose $X$ is isomorphic to its square and complemented in an $M$-space $E$. As $E$ has the same ultra-type as $\ell_\infty$ there is an ultrafilter $\U$ such that $E_\U\sim (\ell_\infty)_\U$ and so $X_\U  \vartriangleleft (\ell_\infty)_\U$. But $X$ is an infinite dimensional $\mathscr L_\infty$-space and so $\ell_\infty$ embeds as a subspace of $X_\U$. Hence
$
\ell_\infty   \vartriangleleft X_\U  \vartriangleleft (\ell_\infty)_\U
$.
Let $\V$ be an ultrafilter such that $(\ell_\infty)_\mathscr V\approx (\ell_\infty)_{\U\times\V}$. One has
$$
(\ell_\infty)_{\U\times \V}\approx (\ell_\infty)_\mathscr V \vartriangleleft X_{\U\times \V}  \vartriangleleft (\ell_\infty)_{\U\times \V}
$$
and since $X$ and $\ell_\infty$ and their ultrapowers are all isomorphic to their squares we can apply Pe\l czy\'nski's
decomposition method again and we are done.
\end{proof}

Regarding the statements quoted in the Introduction, this provides a proof for (c) and amends (a) and (b): both are true (at least) under the additional   hypothesis that $E$ is isomorphic to its square. We show now that Gurari\u\i\ space has a different
ultra-type. Let us recall a few basic facts about
this space. A Banach space
$U$ is said to be of almost-universal disposition if, given isometric embeddings $u: A\to U$ and
$\imath:A\to B$, where $A$ and $B$ are finite dimensional, and $\e>0$,  there is an $(1+\e)$-isometric
embedding $u': B\to U$ such that $u=u'\imath$.
Gurari\u{\i} shows that there exists a separable Banach space of
almost-universal disposition \cite[Theorem 2]{Gurariiold}. This
space was shown by Lusky \cite{luskygura} to be unique, up to
isometries; see \cite{kubis-solecki} for an elementary proof. We will thus call it the Gurari\u{\i} space and denote
it by $\mathcal G$. Henson and Moore \cite[Theorem
6.5]{hensonmoore} show that a Banach space is of almost universal
disposition if and only if some (or every) ultrapower is of almost universal disposition (see \cite[Proposition 5.7]{accgm2} for an improvement of this result).

Gurari\u\i\ space is a Lindenstrauss space and, moreover, every separable Lindenstrauss space is isometric to a
complemented subspace of $\mathcal G$ \cite{wojt} whose complement is
isomorphic to $\mathcal G$ itself \cite{lusk} (see also \cite{p-w}). This implies that $\mathcal G$ is isomorphic (not isometric) to its square and also that $\mathcal G$ is complemented in no $C$-space (Benyamini and Lindenstrauss \cite[Corollary 2]{bl}). With all these prolegomena one has.

\begin{prop}\label{cvsg}
No ultrapower of Gurari\u\i\ space is isomorphic to a complemented subspace of an $M$-space.
\end{prop}
\begin{proof} Assume that  some ultrapower of $\mathcal G$  is isomorphic to a complemented subspace of an $M$-space.
As $\mathcal G$ is isomorphic to its square Theorem~\ref{uplus} implies that there is a compact space $K$, an ultrafilter $\U$ and a linear isomorphism $u:\mathcal G_\U\to C(K)$. Let $G_1$ be a linear subspace of $\mathcal G_\U$ isometric to $\mathcal G$, for instance that lying on the diagonal. Let $A_1$ be the (separable) unital subalgebra that $u(G_1)$ generates in $C(K)$. By Stern result quoted in the proof of Proposition~\ref{Mc0}, there is a separable subspace $G_2$ containing $u^{-1}(A_1)$ having an ultrapower isometric to an ultrapower of $\mathcal G_\U$. This implies that $G_2$ is a space of almost universal disposition. Continuing in this way we get two sequences $(G_n)$ and $(A_n)$ such that:
\begin{itemize}
\item Every $G_n$ is a separable space of almost universal disposition.
\item Every $A_n$ is a separable unital subalgebra of $C(K)$.
\item For every $n\in\mathbb N$ one has $u(G_n)\subset A_n\subset u(G_{n+1})$.
\end{itemize}
Now, letting $G=\overline{\bigcup_{n}G_n}$ and $A=\overline{\bigcup_{n}A_n}$ we see that we $G$ is of almost universal disposition, hence $G\approx \mathcal G$, $A$ is a separable and unital closed subalgebra of $C(K)$, hence a $C$-space, and $u$ is a linear isomorphism from $G$ onto $A$, which contradicts the above mentioned result of Benyamini and
Lindenstrauss.
\end{proof}

This amends the statement quoted as (d) in the Introduction. A more direct proof for this fact
appears in \cite[Theorem 6.1]{accgm2}. It would be however a mistake to think that the
reason for such behaviour is that $\mathcal G$ is not complemented
in any $C$-space, as the following examples show.

\begin{example} (a) There is a (nonseparable) Lindenstrauss space which is complemented in no $C$-space but has an ultrapower isomorphic to a $C$-space.

(b) Under {\sf CH}, there is a separable space that is not even a
quotient of a Lindenstrauss space and has an ultrapower isomorphic to a $C$-space.

\end{example}

\begin{proof} (a) Benyamini constructed in \cite{b} a nonseparable $M$-space which is complemented in no $C$-space. That space has an ultrapower isomorphic to a $C$-space, by Theorem~\ref{uplus}.

(b) It is not hard to check that if $0\to Y \to X\to Z \to 0$ is an
exact sequence and $\mathscr U$ an ultrafilter then  $0\to Y_\U\to
X_\U\to Z_\U\to 0$ is also exact (see \cite[Lemma 2.2.g]{castgonz}).

On the other hand, it
has been shown in \cite[Corollary 2.4]{ccky} that there is an exact sequence
$0\to C(\Delta)\to \Omega\to  C(\Delta)\to 0$ in which $\Omega$ is not even
isomorphic to a quotient of a Lindenstrauss space. Here, $\Delta=2^\N$ is the Cantor set. Let $\U$ be a free ultrafilter on the integers and let us consider the ultrapower sequence
\begin{equation}\label{ul}
\begin{CD} 0@>>>C(\Delta)_\U@>>>\Omega_\U@>>> C(\Delta)_\U@>>> 0.\end{CD}
\end{equation}
We will see that this sequence does split if we assume {\sf CH}. Indeed, Bankston oberved in \cite[Proposition 2.4.1]{bankston}
that, under {\sf CH}, the ultracoproduct $\Delta^\U$ is homeomorphic to $\N^*= \beta \N \setminus \N$, the growth of the integers in its Stone-\v Cech compactification. Thus, under {\sf CH}, the sequence (\ref{ul}) has the form
$0\to C(\N^*)\to \Omega_\U\to  C(\N^*)\to 0$. But we have proved in cite \cite[Proposition 5.6]{accgm1} that every exact sequence of the form $0\to C(\N^*)\to X \to  C(\N^*)\to 0$ splits and so (\ref{ul}) does. Therefore, $\Omega_\U\sim C(\N^*)\times C(\N^*)\approx C(\N^*)$ is a $C$-space.
\end{proof}

\section{Further remarks and open problems}\label{open}

\subsection{More ultratypes, please.}
We have obtained so far only two
different ultra-types of $\mathscr
L_\infty$-spaces: that of $C$-spaces and that of Gurari\u\i\ space.
It would be interesting to add some new classes here. Reasonable
candidates could be the recently constructed hereditarily
indecomposable $\mathscr L_\infty$-spaces \cite{nastiest,tarb}, the preduals of $\ell_1$ in \cite{bl, gasparis}; or
some Bourgain-Pisier spaces \cite{bourpisi}. Since both $\mathcal
G$ and $C$-spaces are Lindenstrauss spaces, one may wonder
whether every $\mathscr L_\infty$-space has an ultrapower
isomorphic to
a Lindenstrauss space.

The following problem was considered by Henson and Moore in \cite[Problem 21]{hensonmoore}. An affirmative answer would imply that the hypothesis of being isomorphic to its square
is superfluous in Theorem~\ref{uplus}.

\begin{problem}
Does every (infinite-dimensional, separable) Banach space $X$ have an ultrapower isomorphic to its square?
What if $X$ is an $\mathscr L_\infty$-space?
\end{problem}

It is perhaps worth noticing that Semadeni proved in \cite{semadeni} that the space of continuous functions on the first uncountable ordinal is not isomorphic to its square. Needless to say, this space has an ultrapower which is isomorphic to its own square.

\subsection{Ultra-splitting}
As we already mentioned, if   $0\to Y \to X\to Z \to 0$ is an
exact sequence and $\mathscr U$ an ultrafilter then  $0\to Y_\U\to
X_\U\to Z_\U\to 0$ is exact again. No
criterion however is known to determine when the ultrapower
sequence of a nontrivial exact sequence splits. Let us say that an
exact sequence ultra-splits if some of its
ultrapower sequences split. Applications of the previous results
yield:

\begin{prop}\label{spliting}  Let $0\to Y \to X\to Z \to 0$ be an exact sequence.
\begin{itemize}  \item If $X$ is a
$C$-space and either $Y$ or $Z$ is the Gurari\u{\i} space, the
sequence does not ultra-split.
\item Under {\sf CH}, if $Y$  an  $\mathscr L_\infty$ space and $Z$ is a separable Banach space complemented in a $C$-space, the
sequence ultra-splits.
\end{itemize}
\end{prop}

\begin{proof}
The first part obviously follows from Proposition~\ref{cvsg}.
As for the second part, we may clearly assume that $Z$ is complemented in $C(\Delta)$. If $\mathscr U$ is a free ultrafilter on $\mathbb N$, then $Z_\mathscr U$ is complemented in $C(\Delta)_\mathscr U$ and, under {\sf CH}, the later space is isometric $C(\N^*)$ which is isometric to $\ell_\infty/c_0$ --in {\sf ZFC}. On the other hand $Y_\mathscr U$ is universally separably injective, by \cite[Theorem 4.10]{accgm1} and so, every sequence $0\to Y_\mathscr U\to E\to \ell_\infty/c_0 \to 0$ splits. Therefore $0\to Y_\U\to
X_\U\to Z_\U\to 0$ splits.
\end{proof}

An interesting case occurs when one puts $\mathcal G$ as the
quotient space. Recall that Johnson and Zippin proved in
\cite{johnzippre} that every separable Lindenstrauss space is a
quotient of $C(\Delta)$; therefore, there exists an exact sequence
\begin{equation}\label{CG}\begin{CD}
0 @>>> \ker q @>>> C(\Delta) @>\pi >> \mathcal G @>>> 0\end{CD}
\end{equation}
which does not ultra-split, by the preceding Proposition. Pe\l czy\'nski posed on
the blackboard to us the question of whether it is possible to
identify the kernel(s) of the preceding sequence(s) and in particular if some kernel can be a $C$-space.
Observe that the structure of $\ker\pi$ effectively depends on the quotient map $\pi$.
It is not hard to check that $\ker\pi$ is
an $\mathscr L_\infty$-space when $\pi$ is an ``isometric" quotient -- this means that $\pi$ maps the open unit ball of $C(\Delta)$ onto that of $\mathcal G$.
On the other hand, Bourgain has shown that $\ell_1$ does contain an uncomplemented subspace isomorphic to itself, from where it follows that there is an exact sequence $0\to E\to F\to G\to 0$ in which both $F$ and $G$ are isomorphic to $c_0$ but $E$ is not an $\mathscr L_\infty$ --this can be seen in \cite[Appendix 1]{BourgainLNM}. Since both $C(\Delta)$ and $\mathcal G$ have (complemented) subspaces isomorphic to $c_0$ we see that there are quotient mappings $\pi:C(\Delta)\to \mathcal G$ whose kernels are not $\mathscr L_\infty$-spaces.

\subsection{Lindenstrauss spaces with isometric ultrapowers}
As we already mentioned, Heinrich undertook in \cite{heinrichL1} the classification of Lindenstrauss spaces up to ultra-isometry.
Amongst the many interesting results he proved one finds that the class of $C$-spaces is closed under ``isometric ultra-roots'': this just means that if a Banach space $X$ has an ultrapower isometric to a $C$-space then $X$ is itself isometric to a $C$-space. A similar result holds for $G$-spaces; see \cite[Theorems 2.7 and 2.10]{heinrichL1}. The result by Henson and Moore \cite[Theorem 6.5]{hensonmoore} that a Banach space is of almost universal
disposition if and only if some (or every) ultrapower is of almost universal disposition shows that the class of Lindenstrauss spaces of almost-universal disposition is also closed under ``isometric ultra-roots''. One can deduce from here that a Banach space $E$ has some ultrapower isometric to an ultrapower of the Gurari\u \i\ space if and only if every separable subspace is contained in a Gurari\u \i\ space contained in $E$.

At the end of \cite{heinrichL1} Heinrich asks whether the classes of $C_0$-spaces and $M$-spaces enjoy the same property.
In a subsequent paper \cite[Section~4]{hhm} (and also in \cite{hensonmoore83}, around Problem 4) it is claimed that there is a Banach space $X$ which fails to be isometric to a Banach lattice and such that $X\stackrel{u}\approx c_0$. Since $c_0$ is both a $C_0$-space and an $M$-space this would imply a negative solution for both questions. Unfortunately, a close inspection to the example reveals that it is indeed a $C_0$-space since it is a subalgebra of $\ell_\infty$. Indeed, if $\mathscr F$ is any almost disjoint family of subsets of $\N$, then the closed linear span of the characteristic functions of the sets of $\mathscr F$ and $c_0$ is always a subalgebra of $\ell_\infty$.  Thus, the following should be considered as an open problem.

\begin{problem}
Are the classes of $C_0$-spaces and $M$-spaces closed under ``isometric ultra-roots''?
\end{problem}

The following problem appears both in \cite{HH} (see Problem 2 on p. 316) and  \cite{hensonmoore83} (see Problems 5 and 7 on pp. 103 and 104).

\begin{problem}[Heinrich, Henson, Moore]
Does Gurari\u\i\ space have an ultrapower isometric (or isomorphic) to an ultraproduct of
finite dimensional spaces?
\end{problem}

Of course the hypothesized  finite dimensional spaces could not be at uniform distance from the corresponding $\ell_\infty^n$ spaces.

\subsection*{Acknowledgements}
AA was supported by MEC and FEDER (Project
MTM-2008-05396), Fundaci\'{o}n S\'{e}neca (Project 08848/PI/08) and
Ram\'on y Cajal contract (RYC-2008-02051). The research of the
other four authors has been supported in part by project
MTM2010-20190. The researh of FCS, JMFC, and YM is included in the
program Junta de Extremadura GR10113 IV Plan Regional I+D+i,
Ayudas a Grupos de Investigaci\'on.

\subsection*{Acknowledgement} The authors want to thank the anonymous referee for his outstanding job.

\end{document}